\documentclass[a4paper,12pt]{amsart}

\makeatletter
\@namedef{subjclassname@2010}{%
  \textup{2010} Mathematics Subject Classification}
\makeatother

\usepackage{setspace,graphicx}     
\usepackage{stackrel}                                          
\usepackage{amssymb}
\usepackage{verbatim} 

\usepackage{blindtext}
\usepackage{tabularx}
\usepackage{longtable}

\usepackage{amsmath} 
\usepackage{amsthm} 
\usepackage{mathrsfs} 
\usepackage{amsfonts} 
\usepackage{latexsym} 
\usepackage{amscd} 
\usepackage[latin1]{inputenc} 
\usepackage[english]{babel} 
\usepackage{enumerate} 
\usepackage{afterpage}
\usepackage{tikz-cd}
\tikzset{commutative diagrams/.cd}
\usetikzlibrary{matrix,arrows,decorations.pathmorphing}
 
\allowdisplaybreaks

\def\f{\mathfrak{f}}

\def\cal{\mathcal}

\def\E{{\cal E}}

\def\p{\mathfrak{p}}

\newcommand{\FF}{\mathbb{F}}

\newcommand{\Z}{\mathbb{Z}} 
\newcommand{\Q}{\mathbb{Q}} 
\newcommand{\PPP}{\mathbb{P}}

\def\Tr{\mathop{\rm Tr}}

\renewcommand{\to}{\rightarrow}

\def\X{\mathcal{X}}

\newtheorem{Theorem}{Theorem}[section]
\newtheorem{Lemma}[Theorem]{Lemma} 
\newtheorem{Conjecture}[Theorem]{Conjecture} 
\newtheorem{Proposition}[Theorem]{Proposition}

\theoremstyle{definition}
\newtheorem{Definition}[Theorem]{Definition}  
\newtheorem{Remark}[Theorem]{Remark} 
\newtheorem{Example}[Theorem]{Example}

\DeclareMathOperator{\Frob}{Frob}

\DeclareMathOperator{\NS}{NS}
\DeclareMathOperator{\Gal}{Gal}
\DeclareMathOperator{\rank}{rank}

\DeclareMathOperator{\End}{End}
\DeclareMathOperator{\Hom}{Hom}
\DeclareMathOperator{\PGL}{PGL}
\DeclareMathOperator{\Fr}{Fr}
\DeclareMathOperator{\USp}{USp}
\DeclareMathOperator{\Sp}{Sp}
\DeclareMathOperator{\U}{U}
\DeclareMathOperator{\Conj}{Conj}
\DeclareMathOperator{\ST}{ST}
\DeclareMathOperator{\GL}{GL}
\DeclareMathOperator{\Mor}{Mor}
\DeclareMathOperator{\Spec}{Spec}

\newtheorem{ind}[]{{\rm\it Indice}}

\begin{document}

\title[The Sato-Tate conjecture and Nagao's conjecture]{The Sato-Tate conjecture and Nagao's conjecture}

\author[S. Kim]{Seoyoung Kim}
\address{Mathematics Department\\Brown University\\
Box 1917, 151 Thayer Street, Providence, RI 02912 USA}
\email{Seoyoung\_Kim@math.brown.edu}

\date{\today}
				
\begin{abstract}
Nagao's conjecture relates the rank of an elliptic surface to a limit formula arising from a weighted average of fibral Frobenius traces, and it is further generalized for smooth irreducible projective surfaces by M. Hindry and A. Pacheco. We show that the Sato-Tate conjecture based on the random matrix model implies Nagao's conjecture for certain twist families of elliptic curves and hyperelliptic curves.   
\end{abstract}



\maketitle

In \cite{NAG}, Nagao suggests a compelling conjecture relating the rank of an elliptic surface $\mathcal{E}$ to a limit formula arising from a weighted average of Frobenius traces from each fiber. Rosen and Silverman \cite{RS} showed Tate's conjecture on the vanishing of $L_{2}(\mathcal{E},s)$ implies Nagao's conjecture. On the other hand, we know from \cite{TATE1} Tate's conjecture implies the Sato-Tate conjecture for elliptic curves which is proven for elliptic curves defined over totally real field or having complex multiplication. In this note, we present some cases which Nagao's conjecture is true assuming the Sato-Tate conjecture. Moreover, using the Sato-Tate conjecture for abelian surfaces \cite{FKRS}, we prove some cases of the generalized Nagao conjecture (following the formulation of Hindry and Pacheco \cite{HP}) for higher genus curves.

\section{Introduction}
\noindent
Let $k/\mathbb{Q}$ be a number field. For a prime $\mathfrak{p}$ of $k$, denote by $\mathbb{F}_{\mathfrak{p}}$ the residue field of $\mathfrak{p}$ and $q_{\p}$ the norm of $\mathfrak{p}$, i.e., $q_{\mathfrak{p}}=\#\mathbb{F}_{\mathfrak{p}}$. Given a smooth projective curve $C$ defined over $k$, let $\mathcal{E}$ be a non-split elliptic surface $\mathcal{E}\rightarrow C$ defined over $k$ which is regular and proper over $C$. For a fixed prime $\mathfrak{p}$, we denote by $\tilde{\mathcal{E}}$ and $\tilde{C}$ over $\mathbb{F}_{\mathfrak{p}}$, respectively, the reductions of $\mathcal{E}$ and $C$ modulo $\mathfrak{p}$. Then there exists a finite set of primes $S$ such that for every prime outside of $S$, the reduced elliptic surface is regular and proper. $S$ could be larger or smaller as needed. Also, for $t \in \tilde{C}(\mathbb{F}_{\mathfrak{p}})$, if the fiber $\tilde{\mathcal{E}_{t}}$ is smooth, define its trace of Frobenius
$$a_{\mathfrak{p}}(\tilde{\mathcal{E}_{t}})=1-\#\tilde{\mathcal{E}_{t}}(\mathbb{F}_{\mathfrak{p}})+ q_{\mathfrak{p}}.$$
We drop the tilde on $\mathcal{E}$ if it is clear in the context. Also, we define
$$A_{\mathfrak{p}}(\mathcal{E})=\frac{1}{q_{\mathfrak{p}}}\sum_{t\in\tilde{C}(\mathbb{F}_{\mathfrak{p}})}a_{\mathfrak{p}}(\mathcal{E}_{t}),$$
we can state Nagao's conjecture \cite{RS}.

\begin{Conjecture}[Nagao's Conjecture over $k$ for elliptic surfaces]
$$\lim_{N\to\infty}\frac{1}{N}\sum_{\substack{\mathfrak{p} \\ q_{\mathfrak{p}}\leq N}}-A_{\mathfrak{p}}(\mathcal{E})\log q_{\mathfrak{p}}=\rank\mathcal{E}(C/k).$$
\end{Conjecture}
\bigskip


In this note, we are going to prove that, in certain cases, the Sato-Tate conjecture for elliptic curves and abelian surfaces imply Nagao's conjecture for elliptic curves and genus $2$ curves, respectively. The following are the main results.

\begin{Theorem}
Let $E$ be an elliptic curve over $k$ with complex multiplication by $F=\mathbb{Q}(\sqrt{-d}),$ i.e., $\End_{\bar{k}}(E)\otimes\mathbb{Q}=F$. Define a (non-split) surface $\E^{f}$ which is a nontrivial quadratic twist of $E: S^2=f(T)=T^3+aT+b$,
$$\E^{f} : y^2=x^3+f(T)^{2}ax+f(T)^{3}b,$$
where $a,b\in k.$ Then for a field $L$ over $k$,\\
(i) 
 $$\rank \E^{f}(L(T)) = \left\{ \begin{array}{ll} 
 1 & \textrm{if $F \not\subset L$}\\
 2 & \textrm{if $F \subset L$.} \end{array} \right. 
$$
(ii) Nagao's conjecture is true for $\E^{f}$.
\end{Theorem}

Moreover, we can prove:

\begin{Theorem}
Let $f(x) \in \mathbb{Q}[x]$ be an monic, squarefree polynomial of degree $3$ or $5$. Suppose that there is an automorphism $\sigma(x)\in \PGL_{2}(\mathbb{Q})$ which permutes the roots of $f$. Suppose furthermore the $\sigma(x)$ does not have a pole at $\infty$. Let $D(T)=f(\frac{T^2}{f(\sigma(\infty))})$. Define a surface $\mathcal{X}_{D}^{f}\rightarrow \PPP^{1}$ by
$$\mathcal{X}_{D}^{f}: D(T)y^{2}=f(x)$$
and assume the Jacobian of its generic fiber $X_{D}^{f}/\Q(T)$ has trivial $\mathbb{Q}(T)/\Q$-trace. Then the following hold.
\begin{enumerate}
\item
If $f(x)$ has degree $3$, Nagao's conjecture is true for $\mathcal{X}^{f}_{D}$.
\item
If $f(x)$ has degree $5$, the Sato-Tate conjecture for abelian surfaces implies Nagao's conjecture for $\mathcal{X}^{f}_{D}$ (Conjecture \ref{HPconj}).
~\\
 Moreover, if $f(x)$ is a self-reciprocal (palindromic) polynomial, we can choose $\sigma(x)=\frac{1}{x}$. Then Nagao's conjecture is true for the surface
$$f(T^2)y^{2}=f(x).$$  
\end{enumerate}
\end{Theorem} 

We can also find the following generalization of the above theorem.

\begin{Theorem}
\label{Thethm}
Let $y^{2}=f(x)$ be an elliptic curve or a hyperelliptic curve of genus $2$ defined over $\Q$. We consider a surface $\mathcal{X}^{f}_{D}\rightarrow \PPP^{1}$ which is defined using a hyperelliptic curve $y^2=D(T)\in\Q[T]$
$$\mathcal{X}^{f}_{D}: D(T)y^{2}=f(x).$$
Denote by $J_{f}$ and $J_{D}$ the Jacobian defined with the curves $y^2=f(x)$ and $y^2=D(T)$ respectively. Also, denote the generic fiber of $\mathcal{X}^{f}_{D}$ by $X^{f}_{D}/\Q(T)$, and we assume that $X^{f}_{D}/\Q(T)$ has trivial $\Q(T)/\Q$-trace.    
Assume that the Jacobian of $D$ has a ($\Q$-isogenous) factorization as $r$ copies of $J_{f}$, i.e., $J_{D}\sim J_{f}^r$. Then the following hold.
\begin{enumerate}
\item
If $y^{2}=f(x)$ is an elliptic curve, Nagao's conjecture is true for $\mathcal{X}^{f}_{D}$.
\item
If $y^{2}=f(x)$ is a hyperelliptic curve of genus $2$, the Sato-Tate conjecture for abelian surfaces implies Nagao's conjecture for $\mathcal{X}^{f}_{D}$ (Conjecture \ref{HPconj}).
\end{enumerate}
\end{Theorem}

\begin{Remark}
The above condition in Theorem \ref{Thethm} on the factorization of the jacobian of a curve is related to many interesting questions. In general, it is not easy to find a hyperelliptic curve $C$ with the jacobian $J(C)$ satisfying
$$J(C)\sim E^{r}\times B$$
for some elliptic curve $E$ and abelian variety $B$ for $r\geq 4$ or with $r=3$ and $\dim B\leq 1$. One can refer \cite[Problem 6.1]{RuSil}. For $r=2$ case, one can find a nice family of curves in Theorem \ref{thm16} and in \cite[Theorem 3.1, Remark 6.3]{RuSil}.
\end{Remark}

\begin{Theorem}
Let $E: y^2=f(x)$ be an elliptic curve over $\Q$. We can consider a surface $\mathcal{X}^{f}_{D}\rightarrow \PPP^{1}$ which is defined using a hyperelliptic curve $y^2=D(T)\in\Q[T]$
$$\mathcal{X}^{f}_{D}: D(T)y^{2}=f(x).$$
Denote by $J_{f}$ and $J_{D}$ the Jacobian defined with the curves $y^2=f(x)$ and $y^2=D(T)$ respectively. Also, denote the generic fiber of $\mathcal{X}^{f}_{D}$ by $X^{f}_{D}/\Q(T)$, and we assume that $X^{f}_{D}/\Q(T)$ has trivial $\Q(T)/\Q$-trace.  
Assume that the Jacobian of $D$ has a ($\Q$-isogenous) factorization $J_{D}\sim E^r\times E_{1} \times E_{2} \times \cdots \times E_{s} $,
where $E_{i}$ are elliptic curves without complex multiplication which are not isogenous to $E$ for $i=1,2,\cdots ,s$. Then Nagao's conjecture is true for $\mathcal{X}^{f}_{D}$.
\end{Theorem}

\section{Nagao's conjecture and the Sato-Tate conjecture for elliptic curves}

In \cite{RS} Rosen and Silverman consider an elliptic surface $\mathcal{E}$ that is a nontrivial twist (by a modular elliptic curve) of an elliptic curve. By proving Tate's conjecture on $\mathcal{E}$ and the nonvanishing of $L_{2}(\mathcal{E}/K,s)$ on $\Re(s)=2$, they showed Nagao's conjecture is true on the twist. We can also prove a special case of Nagao's conjecture just assuming the Sato-Tate conjecture in a more explicit way. Also, we can drop the modularity condition. First, we recall the exact statement of the Sato-Tate conjecture for elliptic curves.

\begin{Conjecture}[Sato-Tate Conjecture]
\label{ST}
Let $E$ be an elliptic curve defined over a number field $k$. For a prime $\mathfrak{p}$ of $k$, denote its absolute norm by $q_{\mathfrak{p}}$. For all but finitely many primes, we can consider the reduction of $E$ at $\mathfrak{p}$, which is again an elliptic curve over the residue field $\mathbb{F}_{\mathfrak{p}}$. And the trace of Frobenius is defined by
\begin{displaymath}
a_{\mathfrak{p}}(E)=1-\#E(\mathbb{F}_{\mathfrak{p}})+q_{\mathfrak{p}}.
\end{displaymath}
From Hasse's theorem, we know $|a_{\mathfrak{p}}(E)|\leq 2\sqrt{q_{\mathfrak{p}}}$. One defines $\theta_{\mathfrak{p}}$ by
$$\frac{a_{\mathfrak{p}}(E)}{2\sqrt{q_{\mathfrak{p}}}}=\cos\theta_{\mathfrak{p}}.$$
The Sato-Tate conjecture tells us the distribution of the trace of Frobenius in terms of the equidistribution with respect to a measure on $[0,\pi]$ of $\theta_{p}$. More precisely,
\begin{flushleft} 
\begin{verse}
\noindent
If $E$ does not have complex multiplication, then the set $\{\theta_{\mathfrak{p}}\}$ is equidistributed with respect to the measure $(2/\pi)\sin^{2}\theta$. Thus, for an interval $[\alpha,\beta]\subset[0,\pi]$,
$$\lim_{X\to \infty}\frac{\#\{\mathfrak{p} : q_{\mathfrak{p}}\leq X,~\alpha\leq\theta_{\mathfrak{p}}\leq\beta\}}{\{\mathfrak{p}:q_{\mathfrak{p}
}\leq X\}}=\frac{2}{\pi}\int^{\beta}_{\alpha}\sin^{2}\theta d\theta.$$
\end{verse} 
\end{flushleft} 
\end{Conjecture}

\begin{Remark}
When $E$ has complex multiplication, the following are already proven. We adopt the same notation of Conjecture \ref{ST}
\begin{enumerate}
\item If $E$ has complex multiplication which is defined over $k$, then the set $\{\theta_{\mathfrak{p}}\}$ is equidistributed with respect to the measure $\frac{1}{\pi} d\theta$.
\item If $E$ has complex multiplication which is not defined over $k$, then $\theta_{\mathfrak{p}}$ is equidistributed with respect to the measure $$\frac{1}{2}\delta_{0}+\frac{1}{2\pi}\frac{dz}{\sqrt{4-z^2}}=\frac{1}{2\pi}d\theta,$$
where $\delta_{0}$ is the Dirac measure at $0$ and $z=\cos\theta$. 
\end{enumerate}
One can find the proof of the above statements in \cite[Corollary 3.8]{FS}.
\end{Remark}

\begin{Remark}
The Sato-Tate conjecture is known to follow from Tate's conjecture \cite{TATE1}. The Sato-Tate conjecture for elliptic curves with complex multiplication is proven relatively easily. Recently, the Sato-Tate conjecture over totally real fields was proven jointly by R. Taylor, L. Clozel, M. Harris, and N. Shepherd-Barron \cite{CHT}\cite{HST}\cite{TAY}. The Sato-Tate conjecture for genus 2, i.e., the limiting behavior of the $L$-fuction of an abelian surface, has been studied by Fit\'e, Kedlaya, Rotger, and Sutherland \cite{FKRS}. The generalized Sato-Tate conjecture in schematic formultation is nicely exposed in Serre's book \cite{SER}. Also, one can find the motivic generalization in \cite{SER1}
\end{Remark}

In this section, we are mostly interested in surfaces which are defined using quadratic twists. For applying the Sato-Tate conjecture in our case, we need the following lemma.

\begin{Lemma}
\label{lem1}
Let $k$ be a number field. Let $\{s_{\mathfrak{p}}\}_{\mathfrak{p}}$ be a bounded sequence of non-negative integers which is defined for each prime ideal $\mathfrak{p}$ of $k$. For a positive integer $N$, define a set of prime ideals with norm less than $N$,
$$T_{N}=\{\mathfrak{p}:~ \text{a prime in $k$ with}~ q_{\mathfrak{p}}< N\}.$$
We also define $\pi_{k}(N)=\#T_{N}$. If one of the following sequences converges
\begin{equation}
\Bigg\{\frac{1}{N}\sum_{\mathfrak{p}\in T_{N}}s_{\mathfrak{p}}\cdot\log q_{\mathfrak{p}}\Bigg\}_{N\geq 1}~~~\text{and}~~~~~~~\Bigg\{\frac{1}{\pi_{k}(N)}\sum_{\mathfrak{p}\in T_{N}}s_{\mathfrak{p}}\Bigg\}_{N\geq 1},
\end{equation}
then both converge to the same limit.
\end{Lemma}
\begin{proof}
The proof is similar to the proof of Nagao \cite{NAG} for $k=\mathbb{Q}$. We can choose a constant $M>0$ such that $|s_{\mathfrak{p}}|<M$ for all prime ideals $\mathfrak{p}$. Moreover, denote the $N$-th term of sequences
$$C_{N}=\sum_{\mathfrak{p}\in T_{N}}s_{\mathfrak{p}}\cdot \log q_{\mathfrak{p}}~~~~~~~~~~D_{N}=\sum_{\mathfrak{p}\in T_{N}}s_{\mathfrak{p}}.$$
From the definition, we get an obvious inequality
$$C_{N} < D_{N}\log N.$$
On the other hand, for $0<\delta<1$, we get the following inequalities
\begin{align*}
C_{N}
 & =\sum_{\mathfrak{p}\in T_{N}}s_{\mathfrak{p}}\cdot \log q_{\mathfrak{p}}\\
 & \geq \sum_{N^{\delta}<q_{\mathfrak{p}}<N}s_{\mathfrak{p}}\log q_{\mathfrak{p}}\\
 & \geq \delta \cdot\log N \sum_{N^{\delta}<q_{\mathfrak{p}}<N}s_{\mathfrak{p}}\\
 & =\delta\cdot \log N \sum_{q_{\mathfrak{p}}<N}s_{\mathfrak{p}}-\delta\cdot\log N \sum_{q_{\mathfrak{p}}\leq N^{\delta}}s_{\mathfrak{p}}\\
 & \geq \delta \cdot \log N \sum_{\mathfrak{p}\in T_{N}}s_{\mathfrak{p}}-\delta\cdot\log N \cdot M\cdot N^{\delta}\\
 & > \delta \cdot \log N \cdot D_{N}-\delta\cdot\log N \cdot M\cdot N^{\delta}.
\end{align*} 
And we know $\pi_{k}(N)=\#T_{N}\sim N/\log(N)$ by the Landau prime ideal theorem. Therefore
\begin{equation}
\limsup_{N\to \infty}\frac{C_{N}}{N}\leq \liminf_{N\to\infty}D_{N}\cdot\frac{\log N}{N}=\liminf_{N \to\infty}\frac{D_{N}}{\pi_{k}(N)}
\end{equation}
and
\begin{equation}
\liminf_{N\to\infty}\frac{C_{N}}{N}\geq \limsup_{N\to\infty}\delta\cdot D_{N}\cdot\frac{\log N}{N}=\limsup_{N\to\infty}\delta \cdot \frac{D_{N}}{\pi_{k}(N)}.
\end{equation}
Thus, the result follows.
\end{proof}

The above lemma enables us to consider a slightly simpler form of Nagao's sum,
$$\lim_{N\to\infty}\frac{1}{\pi_{k}(N)}\sum_{\substack{\mathfrak{p} \\ q_{\mathfrak{p}}\leq N}}-A_{\mathfrak{p}}(\mathcal{E})=\rank\mathcal{E}(C/k).$$
We will write $\pi_{k}(N)=\pi(N)$ if $k$ is defined clearly in the context. In the following, we often consider the above sum instead of the original form. We begin with the following proposition. 

\begin{Proposition}
\label{Prop1}
Let $\E$ be a (non-split) elliptic surface over $\mathbb{Q}$ with $f(T)=T^3+T\in\mathbb{Q}[T]$,
$${\E}: y^2= x^3+f(T)^2x.$$
where $E: S^2=T^3+T$ is an elliptic curve with complex multiplication by $\Z[\sqrt{-1}]$. If we denote
$$A_{p}({\E})=\frac{1}{p}\sum_{t=0}^{p-1}a_{p}(\E_{t}),$$
Nagao's conjecture is true for $\E$, i.e.,
$$\lim_{N \rightarrow \infty}\frac{1}{N}\sum_{p\leq N}-A_{p}(\E)\log p=\rank \E(\Q(T)).$$
\end{Proposition}

\begin{proof} 
We can express each fibral Frobenius trace using the Legendre symbol
$$a_{p}(\E_{t})=\left(\frac{f(t)}{p}\right) \f_{p}$$
and
$$p A_{p}({\E})=\sum_{t=0}^{p-1}a_{p}(\E_{t})=\sum_{t=0}^{p-1}\left(\frac{\f(t)}{p}\right)\f_{p}=-\f_{p}^{2},$$
where $\f_{p}$ is the trace of Frobenius of the elliptic curve $E$ for each prime $p$.

From Conjecture \ref{ST}, we know the distribution of the trace of Frobenius is equidistributed with respect to the uniform measure. Thus   
$$\lim_{N \rightarrow \infty}\frac{1}{\pi(N)}\sum_{p\leq N}-A_{p}(\E)=\lim_{N \rightarrow \infty}\frac{1}{\pi(N)}\sum_{p \leq N}\frac{\f_{p}^2}{p}=\int_{0}^{\pi}4\cos^2\theta\frac{1}{2\pi}d\theta=1.$$
Thus Lemma \ref{lem1} implies
$$\lim_{N \rightarrow \infty}\frac{1}{N}\sum_{p\leq N}-A_{p}(\E)\log p=\lim_{N \rightarrow \infty}\frac{1}{\pi(N)}\sum_{p\leq N}-A_{p}(\E)=1.$$
For verifying Nagao's conjecture, we have to show the rank of $\E(\mathbb{Q}(T))$ is $1$. This proof is based on \cite{NAG}.
Denote $K'=\bar{\mathbb{Q}}(T,S)$, where $T$ and $S$ are variables satisfying $S^2=T^3+T$ which is an elliptic curve having complex multiplication by $\mathbb{Z}[\sqrt{-1}]$. Note that an element $r+s\sqrt{-1} \in \mathbb{Z}[\sqrt{-1}]$ acts on the point $(X,Y)$ as :
$$(X,Y) \longmapsto r\cdot(X,Y)+s\cdot(-X,\sqrt{-1}\cdot Y).$$
We consider the elliptic curve $E : y^2=x^3+x$ defined over $K'=\bar{\mathbb{Q}}(T,S)$ and $C:S^2=T^3+T$. We can rewrite
$$E(K')=\{(X=X(T,S),Y=Y(T,S)) : Y^2=X^3+X\}$$ 
and each point in $E(K')$ define a map
$$(X(T,S),Y(T,S)):C\longrightarrow E.$$
In other words, $K'$-rational points of $E$ are expressed by
\begin{align*}
E(K') & =  \{Maps~~C \rightarrow E\}\\
               & =  \{Maps ~~E \rightarrow E\}\\
							&	=  \{r\cdot(T,S)+s\cdot(-T,\sqrt{-1}\cdot S)+(x_{0},y_{0}) : r,s\in \mathbb{Z},(x_{0},y_{0})\in E(\bar{\mathbb{Q}})\},
\end{align*}
since all morphisms from $E$ to $E$ are the composition of a translation map and an isogeny. The surface $\E$ also can be written by $Y^2=x^3+f(T)^2X$ and thus is $K$-isomorphic to $E$ by the map
$$\phi :(X,Y) \longmapsto (f(T)X, f(T)^{3/2}Y)=(S^2 X, S^3 Y).$$ Define
$$P=\phi(T,S)=(S^2 T, S^4)$$
$$Q=\phi(-T,\sqrt{-1}\cdot S)=(-S^2 T, \sqrt{-1}\cdot S^4).$$
Then all $K'$-rational points of ${\E}$ are given by
$$\{r\cdot P +s \cdot Q +\phi(x_{0},y_{0}) : r,s\in\mathbb{Z}, (x_{0},y_{0})\in E(\bar{\mathbb{Q}})\}.$$ 
Let $\sigma$ be an element of $Gal(K/\bar{\mathbb{Q}}(T))$ such that $S^{\sigma}=-S$. Then we have $P^{\sigma}=P$, $Q^{\sigma}=Q$, and
$$y_{0}=0~~\Leftrightarrow~~\phi(x_{0},y_{0})^{\sigma}=\phi(x_{0},y_{0}).$$
Thus we can write $\bar{\mathbb{Q}}(T)$-rational points of ${\E}$ by
$$\{r \cdot P+s\cdot Q +n_{0}(0,0)+n_{1}(T^2,0)+n_{2}(-T^{2},0) : r,s \in \mathbb{Z}, n_{0},n_{1},n_{2}\in \mathbb{Z}/2{\mathbb{Z}}\}.$$
Therefore we can conclude the rank of $\E(\mathbb{Q}(T))$ is 1.

\end{proof}

Note that we can also observe $\rank\E(\mathbb{Q}(i)(T))$ is $2$ from the above proof. Now we can pose a natural question. How the average of $A_{p}(\E)$ changes if we consider splitting of primes in the ring of integers of $\mathbb{Q}(i)$? The following generalization of Proposition \ref{Prop1} can answer the question.

\begin{Theorem}
\label{thm10}
Let $E$ be an elliptic curve over $k$ with complex multiplication by $F=\mathbb{Q}(\sqrt{-d}),$ i.e., $\End_{\bar{k}}(E)\otimes\mathbb{Q}=F$. Define a (non-split) surface $\E^{f}$ which is a nontrivial quadratic twist of $E: S^2=f(T)=T^3+aT+b$,
$$\E^{f} : y^2=x^3+f(T)^{2}ax+f(T)^{3}b,$$
where $a,b\in k.$ Then for a field $L$ over $k$,\\
(i) 
 $$\rank \E^{f}(L(T)) = \left\{ \begin{array}{ll} 
 1 & \textrm{if $F \not\subset L$}\\
 2 & \textrm{if $F \subset L$.} \end{array} \right. 
$$
(ii) Nagao's conjecture is true for $\E^{f}$.
\end{Theorem}

We need some observations for the proof. Let $\mathscr{H}$ be a surface which is a non-trivial quadratic twist of a fixed hyperelliptic curve $X/k$: There is a double cover $C'\rightarrow C$ and the associated involution $i:C'\rightarrow C'$, then $\mathscr{H}$ is the minimal surface birational to
$$\frac{X\times_{k}C'}{((z,x)\sim(-z,ix))},$$
where $-z$ denotes the involution on $X$ to $z$. We can also describe $\mathscr{H}$ more explicitly as follows. Fix a hyperelliptic curve $X/k$ and let $f(X)\in k[X]$ with distinct roots. Also, let
$$X: Y^{2}=f(X).$$
Then there is a function $g\in k(C)^{*}$ such that $\mathscr{H}$ is a regular proper model for the surface
$$gY^{2}=f(X).$$
We generalize \cite[Proposition 2.3]{RS} for the above surface $\mathscr{H}$. Denote $K=k(C)$ by the function field of $C$ defined over $k$. Let $\mathcal{N}_{\mathscr{H}}$ be the N\'eron model of $\mathscr{H}$ considered as a curve over $K$. Without loss of generality, we assume $C$ is affine. To begin with, we recall the N\'eron mapping property.

\begin{Definition}{(N\'eron mapping property)}
Let $R$ be a Dedekind domain with fractional field $k$, and let $A/k$ be an abelian variety. A N\'eron model $\mathcal{A}$ of $A$ satisfies the following universal property:
\begin{flushleft} 
\begin{verse}
\noindent
Let $\mathcal{X}/R$ be a smooth $R$-scheme with generic fiber $X/k$. For each rational map $\phi_{K}:X_{/k}\rightarrow A_{/k}$ defined over $k$, there exists a unique $R$-morphism $\phi_{R}:\mathcal{X}_{/R}\rightarrow \mathcal{A}_{/R}$ extending $\phi_{k}$.
\end{verse} 
\end{flushleft} 
\end{Definition}

\begin{Proposition}
\label{prop2.3}
Let $\mathscr{H}$ be the twist of $X/k$ using the double cover $C'\rightarrow C$. Denote $J$ and $J'$ by the Jacobian varieties of $C$ and $C'$ respectively. Also, let $J_{X}$ be the Jacobian variety of $X$. We show the following statements.
\begin{enumerate}[(a)]
\item $$\rank \Hom_{k}(J,J_{X})=\rank \Mor_{k}(C,J_{X})/J_X(k)=\rank \NS(J_{X}\times C/k)-2.$$\\
\item Let $i:C'\rightarrow C'$ be the involution associated to the map $C'\rightarrow C'$. Let the eigenspaces
$$\Mor_{k}(C',J_{X})^{\pm}=\{\phi\in\Mor_{k}(C',J_{X}):\phi\circ i=\pm\phi\}.$$
Then we can describe each eigenspace explicitly: Denote $\mathcal{N}_{\mathscr{H}}(C/k)$ its group of sections. 
$$\Mor_{k}(C',J_{X})^{+}\cong\Mor_{k}(C,J_{X})~~~\text{and}~~~\Mor_{k}(C',J_{X})^{-}\cong \mathcal{N}_{\mathscr{H}}(C/k).$$
\item
\begin{align*}
\rank\mathcal{N}_{\mathscr{H}}(C/k) & = \rank \NS(J_{X}\times C'/k)-\rank \NS(J_{X}\times C/k)\\
&= \rank \Hom_{k}(J',J_{X})-\rank \Hom_{k}(J,J_{X}).
\end{align*}
\end{enumerate}
\end{Proposition}

\begin{proof}
For (a), the first equality follows from \cite[Lemma 3.3]{PET}. Also, the proof of \cite[Proposition 2.3(a)]{RS} already implies the second equality. The proof of (b) is also analogous to  \cite[Proposition 2.3(b)]{RS}. A map $\phi:C'\rightarrow J_{X}$ descends to a map $C\rightarrow J_{X}$ if and only if it is constant on each fiber of $C'\rightarrow C$, which is equivalent to $\phi\circ i=\phi$. Thus, $\Mor_{k}(C',J_{X})^{+}\cong\Mor_{k}(C,J_{X})$.
Also, a map $\phi:C'\rightarrow J_{X}$ will descend to a map $C\rightarrow \mathscr{H}$ if and only if $\phi$ respects the equivalence $(z,x)\sim (-z,ix)$, and if and only if $\phi\circ i=-\phi$. On the other hand, the N\'eron mapping property tells the map $C\rightarrow\mathscr{H}$ extends uniquely to a section $\Spec(K[C])\rightarrow \mathcal{N}_{\mathscr{H}}$. Thus, the second isomorphism of (b) holds. Consider the following maps
\begin{align*}
F&:\Mor_{k}(C',J_{X})^{+}\times\Mor_{k}(C',J_{X})^{-}\rightarrow\Mor_{k}(C',J_{X}), ~~~F(\phi_{1},\phi_{2})=\phi_{1}+\phi_{2},\\
G&:\Mor_{k}(C',J_{X})\rightarrow\Mor_{k}(C',J_{X})^{+}\times\Mor_{k}(C',J_{X})^{-},~~~G(\phi)=(\phi+i(\phi),\phi-i(\phi)),
\end{align*}
where $i$ is the natural involution on $\Mor_{k}(C',J_{X})$ induced from the involution on $C'$. Then we have
$$F(G(\phi))=2\phi,~~~\text{and}~~~G(F(\phi_{1},\phi_{2}))=2(\phi_{1},\phi_{2}).$$ 
Thus, $\Mor_{k}(C',J_{X})$ is isomorphic to the direct sum of $\Mor_{k}(C',J_{X})^{+}$ and $\Mor_{k}(C',J_{K})^{-}$ up to $2$-torsion. Therefore, (c) follows from (a) and (b). 
\end{proof}

Thus, by the construction of the N\'eron model, we have the following isomorphism
\begin{align*}
\mathcal{N}_{\mathscr{H}}(C/K) &= \{\text{sections}~\Spec(K[C])\rightarrow \mathcal{N}_{\mathscr{H}}\}\\
&=\mathcal{N}_{\mathscr{H}}(K[C])\\
&\cong J_{X}(K).
\end{align*}
Therefore, we have the following equality of ranks:
\begin{equation}
\rank J_{X}(K)=\rank \Hom_{k}(J',J_{X})-\rank \Hom_{k}(J,J_{X}).
\end{equation}
Taking $C'=X$ and $C=\mathbb{P}^{1}$, this implies
\begin{equation}
\rank J_{X}(K)=\rank \End_{k}(J_{X}).
\end{equation}

We resume the proof of Theorem \ref{thm10}.

\begin{proof}[Proof of Theorem \ref{thm10}]
Given a prime ideal $\mathfrak{p}$ of $k$, denote the trace of Frobenius of $E$ by $\f_{\mathfrak{p}}$. Since we are working on the quadratic twist of $E$, for fixed $T=t$, we have
$$a_{\mathfrak{p}}(\E^{f}_{t})=\chi_{\mathfrak{p}}(f(t))\cdot\f_{\mathfrak{p}},$$
where $\chi_{\mathfrak{p}}:\mathbb{F}_{\mathfrak{p}}^{*}\rightarrow\{1,-1\}$ is the unique nontrivial character of order $2$. Naturally, define $\chi_{\mathfrak{p}}(0)=0$, and
\begin{align*}
\f_{\p} & =q_{\p}+1-\#E(\FF_{\p})\\
&=q_{\p}+1-\left(1+\sum_{t\in\FF_{\p}}1+\chi_{\p}(f(t))\right)\\
&=-\sum_{t\in\FF_{\p}}\chi_{\p}(f(t)),
\end{align*}
which shows
$$A_{\mathfrak{p}}(\E^{f})=\frac{1}{q_{\mathfrak{p}}}\sum_{t\in\mathbb{F}_{\mathfrak{p}}}a_{\mathfrak{p}}(\E^{f}_{t})=\frac{1}{q_{\mathfrak{p}}}\sum_{t\in\mathbb{F}_{\mathfrak{p}}}\chi_{\mathfrak{p}}(f(t))\cdot\f_{\mathfrak{p}}=-\frac{\f^{2}_{\mathfrak{p}}}{q_{\mathfrak{p}}}.$$
Thus we get
$$\lim_{N \to \infty}\frac{1}{\pi(N)}\sum_{q_{\mathfrak{p}}\leq N}-A_{\mathfrak{p}}(\E^{f})=\lim_{N \to \infty}\frac{1}{\pi(N)}\sum_{q_{\mathfrak{p}} \leq N}\frac{\f^{2}_{\mathfrak{p}}}{q_{\mathfrak{p}}}.$$
Recall that if $E$ has complex multiplication defined over $L$, $\theta_{\mathfrak{p}}=\arccos\frac{a_{\mathfrak{p}}}{2\sqrt{q_{\mathfrak{p}}}}$ is uniformly distributed on $[0,\pi]$. On the other hand, if $E$ has complex multiplication not defined over $L$, $\theta_{\mathfrak{p}}=\arccos\frac{a_{\mathfrak{p}}}{2\sqrt{q_{\mathfrak{p}}}}$ is half uniformly distributed on $[0,\pi]$ and takes discrete measure of mass $\frac{1}{2}$ at $\frac{\pi}{2}$. Thus, if $F\subset L$, we can do the calculation
$$\lim_{N \to \infty}\frac{1}{\pi(N)}\sum_{q_{\mathfrak{p}}\leq N}-A_{\mathfrak{p}}(\E^{f})=\lim_{N \to \infty}\frac{1}{\pi(N)}\sum_{q_{\mathfrak{p}} \leq N}\frac{\f^{2}_{\mathfrak{p}}}{q_{\mathfrak{p}}}=\int^{\pi}_{0}4\cos^{2}\theta \frac{1}{\pi}d\theta=2.$$
Similarly we get $1$ for the case $F\not\subset L$. To finish the proof, we need to show $\rank \E^{f}(L(T))$ agrees to $2$ and $1$ in each case. From Proposition \ref{prop2.3}, we have 
$$\rank \E^{f}(L(T))=\rank \End_{L}(E).$$
Also we need the following theorem from \cite[II.2.2.]{SIL2}.

\begin{Theorem}
Let $E$ be an elliptic curve defined over a field $L \subset \mathbb{C}$ and with complex multiplication by the quadratic imaginary field $K \subset \mathbb{C}$. Then every endomorphism of $E$ is defined over the compositum $LK$.
\end{Theorem}

Thus if $F \subset L$ where $F$ is the complex multiplication field of $E$, we have $\rank E^{f}(L(T))=2$, and $1$ otherwise. Thus we showed Nagao's conjecture holds for the surface $\E^{f}$.
\end{proof}

\section{Nagao's conjecture and the Sato-Tate conjecture for abelian surfaces}

We can generalize Nagao's conjecture for elliptic surfaces to smooth projective surfaces over a number field. The formulation of Nagao's conjecture for smooth projective surfaces follows \cite{HP}.

\subsection{Nagao's conjecture for smooth irreducible projective surfaces}
~\\
Let $k/\mathbb{Q}$ be a number field. Let $\mathcal{X}$ be a smooth irreducible projective surface over $k$ and $C$ be a smooth irreducible projective curve over $k$ which allows a proper flat morphism $f:\mathcal{X}\rightarrow C$ so that the fibers are curves of (arithmetic) genus $g\geq 1$. The assumption on $f:\mathcal{X} \rightarrow C$ implies the irreducibility and smoothness of the generic fiber. Denote the generic fiber of $f$ by $X/K$, where $K=k(C)$ is the function field of $C$. Also, denote $J(X)$ the Jacobian of $X$ and $(B,\tau)$ the $K/k$-trace of $J(X)$. The definition of $K/k$-trace follows S. Lang \cite[p.138]{LANG}.
\begin{Definition}
Let $k$ be a number field and $K$ be a function field defined over $k$. A {\it $K/k$-trace} of an abelian variety $A$ over $K$ is a pair $(B,\tau)$ with an abelian variety $B$ over $k$ and a homomorphism $\tau: B\rightarrow A$ over $K$ with the following universal mapping property: For an abelian variety $C$ over $k$ and a homomorphism $\varphi: C \rightarrow A$ defined over $K$, there exists a unique homomorphism $\bar{\varphi}: C\rightarrow B$ defined over $k$ which makes the following diagram commutes.
\begin{center}
\begin{tikzcd} 
C \arrow{r}{\varphi} \arrow{d}{\bar{\varphi}} 
&A \\ 
B \arrow{ur}{\tau}
\end{tikzcd}
\end{center}
\end{Definition}
We know $J(X)(K)/\tau(B(k))$ is finitely generated by the theorem of S. Lang and A. N\'eron \cite{LN}\cite{NER}.\\
\\
For a prime ideal $\mathfrak{p}$ in $k$, we can consider the reduction $f_{\mathfrak{p}}$ of $f$. Define a finite set of prime ideal $S$ which satisfies the following conditions: For all $\mathfrak{p}\notin S$, $\mathcal{X}$ and $C$ have good reduction and 
$$f_{\mathfrak{p}}:\mathcal{X}_{\mathfrak{p}}\rightarrow C_{\mathfrak{p}}$$
is proper and flat. Also, the fibers are curves with arithmetic genus $g$ over the residue field $\mathbb{F}_{\mathfrak{p}}$. Denote $\mathcal{X}_{\mathfrak{p},c}=f_{\mathfrak{p}}^{-1}(c)$, i.e., the fiber of $f_{\mathfrak{p}}$ at $c\in C_{\mathfrak{p}}$. Denote $\overline{\mathcal{X}}_{\mathfrak{p},c}=\mathcal{X}_{\mathfrak{p},c}\times_{k}\bar{k}$, where $\bar{k}$ is the algebraic closure of $k$. Define $G_{k}=\Gal(\bar{k}/k)$ be an absolute Galois group of $k$.\\
\\
Let $\Fr_{\mathfrak{p}}\in G_{k}$ be a Frobenius element and $I_{\mathfrak{p}}\subset G_{k}$ be the inertia group. Also, define the discriminant of $f$, $\Delta=\{c\in C:\mathcal{X}_{\mathfrak{p},c}~\text{is singular}\}$. By enlarging the set $S$, we can make the discriminant of $f_{\mathfrak{p}}$ is the same as the discriminant of $f$ modulo $\mathfrak{p}$ outside of $S$.\\
\\
Let $\overline{\Fr}_{\mathfrak{p}}$ be the Frobenius automorphism on $H^{1}_{\acute{e}t}(\overline{\mathcal{X}}_{\mathfrak{p},c},\mathbb{Q}_{l})$. Define the trace of Frobenius using cohomology
$$a_{\mathfrak{p}}(\mathcal{X}_{\mathfrak{p},c})=\Tr(\overline{\Fr}_{\mathfrak{p}}|H^{1}_{\acute{e}t}(\overline{\mathcal{X}}_{\mathfrak{p},c},\mathbb{Q}_{l})),$$
where we consider cohomology with proper support if $c\in \Delta_{\mathfrak{p}}(\mathbb{F}_{\mathfrak{p}})$. Also, define
$$a_{\mathfrak{p}}(B)=\Tr(\overline{\Fr}_{\mathfrak{p}}|H^{1}_{\acute{e}t}(\overline{B},\mathbb{Q}_{l})^{I_{\mathfrak{p}}}).$$
By enlarging the set $S$ if necessary, we can assume $B$ has good reduction for primes $\mathfrak{p}\notin S$, i.e., 
$$a_{\mathfrak{p}}(B)=\Tr(\overline{\Fr}_{\mathfrak{p}}|H^{1}_{\acute{e}t}(\overline{B},\mathbb{Q}_{l})).$$
Hindry and Pacheco \cite{HP} generalize Conjecture \ref{ST} for smooth irreducible projective surfaces as follows.

\begin{Conjecture}[Nagao's conjecture for smooth irreducible projective surfaces, \cite{HP}]
\label{HPconj}
Define the average trace of Frobenius
$$A_{\mathfrak{p}}(\mathcal{X})=\frac{1}{q_{\mathfrak{p}}}\sum_{c\in C_{\mathfrak{p}}(\mathbb{F}_{\mathfrak{p}})}a_{\mathfrak{p}}(\mathcal{X}_{\mathfrak{p},c}),$$
and
$$A_{\mathfrak{p}}^{*}(\mathcal{X})=A_{\mathfrak{p}}(\mathcal{X})-a_{\mathfrak{p}}(B).$$
Then the following equality holds
$$\lim_{N\to \infty}\frac{1}{N}\sum_{\substack{\mathfrak{p}\notin S \\ q_{\mathfrak{p}\leq X}}}-A_{\mathfrak{p}}^{*}(\mathcal{X})\log q_{\mathfrak{p}}=\rank(J(X)(K)/\tau B(k)).$$
\end{Conjecture}

~\\
In this section, we are mostly interested in the case when $\mathcal{X}$ is a surface with genus $2$ fibers with $C=\mathbb{P}^{1}$. The Sato-Tate conjecture for the application of Conjecture \ref{HPconj} for the above $\mathcal{X}$ is much more complicated from Conjecture \ref{ST} of the previous section. Unlike Conjecture \ref{ST}, there are $52$ possibilities of different Sato-Tate distributions which corresponds to the appearance of $52$ different Sato-Tate groups up to conjugacy. To begin with, we state the main result of this section.

\begin{Remark}
From now on, we assume $B$ is trivial which simplifies Conjecture \ref{HPconj}. Using the above notation, roughly, since we consider fields with zero characteristic, the $K/k$-trace is the largest subvariety of $J(X)$ that can be defined over $k$. For instance, the $K/k$-trace of a non-constant elliptic curve is trivial \cite[Example 2.2]{CON}. For more information, one can refer to \cite{CON}\cite{HP} and \cite{LANG}.
\end{Remark}

\begin{Theorem}
\label{thm2}
Let $y^{2}=f(T)$ be a hyperelliptic curve of genus $2$ defined over $\Q$, that is to say, $f(T)\in\Q[T]$ is a monic, square-free polynomial with degree $5$ or $6$. Define a surface $\mathcal{X}^{f}\rightarrow \mathbb{P}^{1}$:
$$\mathcal{X}^{f} : f(T)y^2=f(x),$$
and suppose the Jacobian of its generic fiber $X^{f}/\Q(T)$ has trivial $\Q(T)/\Q$-trace. Then the Sato-Tate conjecture for abelian surfaces implies Nagao's conjecture for $\mathcal{X}^{f}$. More precisely, if $\mathcal{X}^{f}_{t}$ denotes the genus $2$ curve at $T=t$, define the average trace of Frobenius
$$A_{p}(\mathcal{X}^{f})=\frac{1}{p}\sum_{t=0}^{p-1}a_{p}(\mathcal{X}^{f}_{t}).$$
Then the following equality holds.
$$\lim_{N \to \infty}\frac{1}{\pi(N)}\sum_{p<N}-A_{p}(\mathcal{X}^{f})=\rank J(X^{f})(\mathbb{Q}(T)).$$
\end{Theorem}

\subsection{The Sato-Tate conjecture for an abelian surface}
~\\
We state briefly the Sato-Tate conjecture for an abelian surface. The presented formulation is from \cite{FKRS}. For a genus $2$ curve $C$, we know from the Weil conjectures for its Jacobian $A=J(C)$, there is a zeta function for each $p$ which can be factored in the following way,
$$Z(A/\mathbb{F}_{p},T)=\frac{1+a_{p}T+b_{p}T^{2}+pa_{p}T^{3}+p^{2}T^{4}}{(1-T)(1-pT)}.$$
We denote the numerator of $Z(A/\mathbb{F}_{p},T)$ by $L_{p}(A,T)$. It is known that $L_{p}(A,T)=\det(1-T\Frob_{p}, V_{l}(A))$ where $\Frob_{p}$ is an arithmetic Frobenius element of $G_{k}=\Gal(k^{sep}/k)$ acting on the (rational) $l$-adic Tate module of $A$, $V_{l}(A)=\mathbb{Q}\otimes T_{l}(A)$ and define the normalized $L$-polynomial $\bar{L}_{p}(A,T)=L_{p}(A,q^{-\frac{1}{2}}T).$ Then we know the roots of $\bar{L}_{p}(A,T)$ have norm $1$ and are stable under complex conjugation as a set. Thus $\bar{L}_{p}(A,T)$ corresponds to a unique element in the unitary symplectic group $\USp(4)$ up to conjugacy: $\bar{L}_{p}(A,T)$ is the characteristic polynomial of the corresponding matrix.
$$\USp(4)=\U(4)\cap \Sp(4,\mathbb{C}),$$
where $\U(4)$ is the group of unitary matrices and $\Sp(4,\mathbb{C})$ is the symplectic group of degree $4$ over $\mathbb{C}$, that is to say, the group of $4\times 4$ symplectic matrices with entries in $\mathbb{C}$.

In \cite{KS}, Katz and Sarnak conjectured the equidistribution of $\bar{L}_{\p}(A,T)$ (also, the equidistribution of trace of Frobenius which appears as the coefficient of $L_{\p}(A,T)$) with respect to a measure arising from the above correspondence. More precisely, $L_{\p}(A,T)$ are equidistributed with respect to the image of the normalized Haar measure on $\Conj(\USp(4))$ which arises from a suitable closed subgroup $G$ of $\USp(4)$. The carefully selected closed group $G\subset \USp(4)$ is called the Sato-Tate group of $A$ and is denoted by $\ST_{A}$. The exact definition can be found in \cite[Definition 2.6]{FKRS}. Using the Sato-Tate group, we can phrase the Sato-Tate conjecture for genus $2$.

\begin{Conjecture}[Refined Sato-Tate for genus 2, \cite{FKRS}]
\label{refined}
For an abelian surface $A$ and its Sato-Tate group $\ST_{A}$, we define $\mu_{\ST_{A}}$ to be the image on $\Conj(\ST_{A})$ of the normalized Haar measure on $\ST_{A}$. Then the classes $s(\p)\in \Conj(\ST_{A})$ are equidistributed with respect to $\mu_{\ST_{A}}$.
\end{Conjecture}

The refined form of Sato-Tate implies the following.
\begin{enumerate}
    \item From the above identification, the $L$-polynomials of $A$, $\bar{L}_{\p}(A,T)$ in $\Conj(\USp(4))$, are equidistributed with respect to the image of the Haar measure for $\ST_{A}$.
		\item By using the definition of Sato-Tate group, one can get the equidistribution of $\bar{L}_{\p}(A,T)$ in $\Conj(\ST_{A})$. See \cite{FKRS} for details.
		\item One can relate $\ST_{A}$ to the endomorphism ring of $A$. See \cite{FKRS} for details.
\end{enumerate}

One can write down all the possible Sato-Tate groups for genus $2$. It turns out that there are exactly $52$ possible conjugacy classes of the Sato-Tate group for an abelian surface $A$. Moreover, we know that only $34$ of them can arise when $k=\mathbb{Q}$. Table \ref{my-label}, from \cite[Table 8, Table 11]{FKRS}, shows the possible Sato-Tate groups and related information. In this table, the third column represents $E[a_{p}^{2}/p]$, the second moment of the trace of Frobenius of given curve with the fixed Sato-Tate group in the first column. The fourth column gives examples of curves of prescribed Sato-Tate groups.

\begin{table}[tbp]
\centering
\caption{Possible Sato-Tate group over $\mathbb{Q}$ from \cite[Table 8, Table 11]{FKRS}}
\label{my-label}
\begin{tabular}{|c|c|c|l|}
\cline{1-4}
$ST_{J(C)}$ & $\End (J(C))_{\mathbb{R}}$     & $E[a_{p}^{2}/p]$  & Example   \\ \cline{1-4} 
$J(C_{2})$  & $\mathbb{C}$                   & 2        & $y^2=x^5-x$       \\ \cline{1-4}    
$J(C_{4})$  & $\mathbb{C}$                   & 2        & $y^2=x^6+x^5-5x^4-5x^2-x+1$       \\ \cline{1-4}             
$J(C_{6})$  & $\mathbb{C}$                   & 2        & $y^2=x^6-15x^4-20x^3+6x+1$       \\ \cline{1-4}  
$J(D_{2})$  & $\mathbb{R}$                   & 1        & $y^2=x^5+9x$       \\ \cline{1-4}   
$J(D_{3})$  & $\mathbb{R}$                   & 1        & $y^2=x^6+10x^3-2$       \\ \cline{1-4}  
$J(D_{4})$  & $\mathbb{R}$                   & 1        & $y^2=x^5+3x$       \\ \cline{1-4}   
$J(D_{6})$  & $\mathbb{R}$                   & 1        & $y^2=x^6+3x^5+10x^3-15x^2+15x-6$       \\ \cline{1-4}   
$J(T)$      & $\mathbb{R}$                   & 1        & $y^2=x^6+6x^5-20x^4+20x^3-20x^2-8x+8$       \\ \cline{1-4}  
$J(O)$      & $\mathbb{R}$                   & 1        & $y^2=x^6-5x^4+10x^3-5x^2+2x-1$       \\ \cline{1-4}   
$C_{2,1}$   & $M_{2}(\mathbb{R})$            & 4        & $y^2=x^6+1$       \\ \cline{1-4}    
$C_{6,1}$   & $\mathbb{C}$                   & 2        & $y^2=x^6+6x^5-30x^4+20x^3+15x^2-12x+1$       \\ \cline{1-4}     
$D_{2,1}$   & $\mathbb{R}\times \mathbb{R}$  & 2        & $y^2=x^5+x$       \\ \cline{1-4}    
$D_{4,1}$   & $\mathbb{R}$                   & 1        & $y^2=x^5+2x$       \\ \cline{1-4}     
$D_{6,1}$   & $\mathbb{R}$                   & 1        & $y^2=x^6+6x^5-30x^4-40x^3+60x^2+24x-8$       \\ \cline{1-4}     
$D_{3,2}$   & $\mathbb{R}\times \mathbb{R}$  & 2        & $y^2=x^6+4$       \\ \cline{1-4}   
$D_{4,2}$   & $\mathbb{R}\times\mathbb{R}$   & 2        & $y^2=x^6+x^5+10x^3+5x^2+x-2$       \\ \cline{1-4}   
$D_{6,2}$   & $\mathbb{R}\times\mathbb{R}$ & 2          & $y^2=x^6+2$      \\ \cline{1-4}   
$O_{1}$      & $\mathbb{R}$                   & 1       & $y^2=x^6+7x^5+10x^4+10x^3+15x^2+17x+4$       \\ \cline{1-4}
$E_{1}$      & $M_{2}(\mathbb{R})$            & 4       & $y^2=x^6+x^4+x^2+1$       \\ \cline{1-4}  
$E_{2}$      & $\mathbb{C}$                   & 2       & $y^2=x^6+x^5+3x^4+3x^2-x+1$       \\ \cline{1-4}
$E_{3}$      & $\mathbb{C}$                   & 2       & $y^2=x^5+x^4-3x^3-4x^2-x$       \\ \cline{1-4} 
$E_{4}$      & $\mathbb{C}$                   & 2       & $y^2=x^5+x^4+x^2-x$       \\ \cline{1-4}  
$E_{6}$      & $\mathbb{C}$                   & 2       & $y^2=x^5+2x^4-x^3-3x^2-x$       \\ \cline{1-4}  
$J(E_{1})$   & $\mathbb{R}\times\mathbb{R}$   & 2       & $y^2=x^5+x^3+x$       \\ \cline{1-4} 
$J(E_{2})$   & $\mathbb{R}$                   & 1       & $y^2=x^5+x^3-x$       \\ \cline{1-4} 
$J(E_{3})$   & $\mathbb{R}$                   & 1       & $y^2=x^6+x^3+4$       \\ \cline{1-4} 
$J(E_{4})$   & $\mathbb{R}$                   & 1       & $y^2=x^5+x^3+2x$       \\ \cline{1-4}
$J(E_{6})$   & $\mathbb{R}$                   & 1       & $y^2=x^6+x^3-2$      \\ \cline{1-4}
$F_{ac}$     & $\mathbb{R}$                   & 1       & $y^2=x^5+1$       \\ \cline{1-4} 
$F_{a,b}$    & $\mathbb{R} \times \mathbb{R}$ & 2       & $y^2=x^6+3x^4+x^2-1$       \\ \cline{1-4}
$N(G_{1,3})$ & $\mathbb{R}\times\mathbb{R}$   & 2       & $y^2=x^6+3x^4-2$       \\ \cline{1-4}
$G_{3,3}$    & $\mathbb{R}\times \mathbb{R}$  & 2       & $y^2=x^6+x^2+1$       \\ \cline{1-4} 
$N(G_{3,3})$ & $\mathbb{R}$                   & 1       & $y^2=x^6+x^5+x-1$       \\ \cline{1-4} 
$\USp(4)$       & $\mathbb{R}$                   & 1       & $y^2=x^5-x+1$    \\ \cline{1-4} 

\end{tabular}
\end{table}

Now we can resume proof of Theorem \ref{thm2}.

\begin{proof}[Proof of Theorem \ref{thm2}]
Denote by $\f_{p}$ the trace of Frobenius of the genus $2$ hyperelliptic curve $y^{2}=f(X)$. For fixed $T=t$, we have 
$$a_{p}(\mathcal{X}^{f}_{t})=\left(\frac{f(t)}{p}\right) \f_{p}$$
and we know
$$A_{p}(\mathcal{X}^{f})=\frac{1}{p}\sum_{t=0}^{p-1}a_{p}(C^{f}_{t})=\frac{1}{p}\sum_{t=0}^{p-1}\left(\frac{f(t)}{p}\right) \f_{p}=-\frac{\f_{p}^{2}}{p}.$$
Therefore 

\begin{displaymath}
\lim_{N \to \infty}\frac{1}{\pi(N)}\sum_{p<N}-A_{p}(\mathcal{X}^{f})=\lim_{N \to \infty}\frac{1}{\pi(N)}\sum_{p<N}\frac{\f^{2}_{p}}{p}=
\left(
\begin{tabular}{c} Second moment of normalized \\ trace of Frobenius $\f_{p}/\sqrt{p}$
\end{tabular} \right).
\end{displaymath}
We know, from the above classification, the Jacobian $J_{f}$ of $y^2=f(x)$ has one of the groups above as its the Sato-Tate group. And one can easily read off the second moment of the normalized trace of Frobenius from the third column of Table \ref{my-label}. 
On the other hand, we know from Proposition \ref{prop2.3},
$$\rank J(X^{f})(\mathbb{Q}(T))=\rank \End_{\mathbb{Q}}(J_{f})=\rank_{\mathbb{R}} (\End_{\mathbb{Q}}(J_{f})\otimes_{\mathbb{Z}}{\mathbb{R}}).$$
Again, the value $\rank_{\mathbb{R}} (\End_{\mathbb{Q}}(J_{f})\otimes_{\mathbb{Z}}{\mathbb{R}})$ is the $\mathbb{R}-$rank of the endomorphism ring on the second column of Table \ref{my-label} for the given Sato-Tate group of $y^2=f(x)$. We can see the coincidence of these two different values for each Sato-Tate group, and thus we get Conjecture \ref{HPconj} for each Sato-Tate group.
\end{proof}

\begin{Example}{(Case $\ST_{J(C)}=\USp(4)$)}
For a genus $2$ curve $C$, we know its Jacobian $A=J(C)$ has $L$-polynomial 
$$L_{p}(A,T)=1+a_{p}T+b_{p}T^{2}+pa_{p}T^{3}+p^{2}T^{4}$$
and we can normalize it by $\bar{L}_{p}(A,T)=L_{p}(A,q^{-\frac{1}{2}}T)$.
As we explained before, we know the roots of $\bar{L}_{p}(A,T)$ have norm $1$ and are stable under complex conjugation as a set. 
If we represent $\frac{1}{\sqrt{p}}\Frob_{p}$ as a matrix, it is conjugate to the matrix
\[
\begin{bmatrix}
    e^{i\theta_{1}} &  &  &  \\
                    & e^{i\theta_{2}} &  & \\
                    &                 & e^{-i\theta_{1}} &  \\
     & &  & e^{-i\theta_{2}} 
\end{bmatrix}
\] 
which is in $\USp(4)$. We can write
$$\det(1-T\Frob_{p}, V_{1}(A))=(1-T\sqrt{p}e^{i\theta_{1}})(1-T\sqrt{p}e^{i\theta_{2}})(1-T\sqrt{p}e^{-i\theta_{1}})(1-T\sqrt{p}e^{-i\theta_{2}}).$$
Thus we have 
$$a_{p}=\sqrt{p}(2\cos\theta_{1}+2\cos\theta_{2}).$$
Let's compute explicitly the case when the Sato-Tate group $\ST_{J(C)}$ is $\USp(4)$. We know an explicit Haar measure on $\USp(4)$,
$$\frac{8}{\pi^{2}}(\cos\theta_{1}-\cos\theta_{2})^{2}\sin^{2}\theta_{1}\sin^{2}\theta_{2}~d\theta_{1}d\theta_{2}.$$
Thus the second moment of the linear coefficient of the normalized $L$-polynomial is 
$$4\int_{0}^{\pi}\int_{0}^{\pi}(\cos\theta_{1}+\cos\theta_{2})^{2}\cdot \frac{8}{\pi^{2}}(\cos\theta_{1}-\cos\theta_{2})^{2}\sin^{2}\theta_{1}\sin^{2}\theta_{2}~d\theta_{1}d\theta_{2}=1,$$
which means 
$$\lim_{N \to \infty}\frac{1}{\pi(N)}\sum_{p<N}-A_{p}(C)=\lim_{N \to \infty}\frac{1}{\pi(N)}\sum_{p<N}\frac{a^{2}_{p}}{p}=1.$$
On the other hand, from the last line of Table~\ref{my-label} and Theorem \ref{thm2}, this is equal to the rank of $J(C)(\mathbb{Q}(T))$.
\end{Example}

Thus we have shown Nagao's conjecture for the twist 
$$f(T)y^{2}=f(x),$$
where $f(x)\in\mathbb{Q}[x]$ is a polynomial of degree $3$ or $5$. How much can we see when the twist has a little different form? Actually we can prove that the Sato-Tate conjecture implies Nagao's conjecture for twists of the form
$$f(T^{2})y^2=f(x)$$
with $f(x)\in\mathbb{Q}[x]$ of degree $3$ or $5$. For proving this, let's first introduce a theorem of S. Peterson \cite{PET}. Denote $k^{*}=k\setminus \{0\}$. Recall the projective linear group
$$\PGL_{2}(k):=\GL_{2}(k)/k^{*}$$
which may be identified with the automorphism group of $\PPP^{1}$, i.e.,
\begin{displaymath}
\left(\begin{array}{cc} a & b \\  c & d \end{array}\right): X\mapsto \frac{aX+b}{cX+d}
\end{displaymath}
\begin{Theorem}
\label{thm16}
Let $k$ be a number field. Let $f(x)\in k[x]$ be a monic, squarefree polynomial of odd degree $d$. Suppose that there is an automorphism $\sigma(x)\in \PGL_{2}(k)$ which permutes the roots of $f$. Suppose furthermore that $\sigma(x)$ does not have a pole at $\infty$. Let 
$$D(T)=f\left(\frac{T^{2}}{f(\sigma(\infty))}+\sigma^{-1}(\infty)\right).$$
Denote by $J_{f}$, $J_{D}$, and $J^{D}_{f}$ the Jacobian defined with the curves $y^{2}=f(x)$, $y^{2}=D(T)$, and $D(T)y^{2}=f(x)$ respectively. Also, denote by $a_{\p}(J_{D})$ and $a_{\p}(J_{f})$ the linear coefficient of $L_{\p}(J_{D},T)$ and $L_{\p}(J_{f},T)$ respectively. Then we have 

\begin{enumerate}
\item There is a $k$-isogeny $J_{D} \sim J_{f}\times J_{f}$.
\item We have $\rank J^{D}_{f}(k(T))=2\rank \End_{k}(J_{f})$.
\item The $L$-polynomial of $J_{D}$ factors nicely, $L_{\p}(J_{D},T)=L_{\p}(J_{f},T)^{2}$, i.e., the linear coefficients satisfy
$$a_{\p}(J_{D})=2a_{\p}(J_{f}).$$
\end{enumerate}
\end{Theorem}

\begin{proof}
(1) and (2) follows from \cite[Theorem 7.4]{PET}. Isogenous abelian varieties have the same $L$-function, thus (3) follows.
\end{proof}

For the next theorem, we adopt the notation from Theorem \ref{thm16}.

\begin{Theorem}
\label{thm20}
Let $f(x) \in \mathbb{Q}[x]$ be an monic, squarefree polynomial of degree $3$ or $5$. Suppose that there is an automorphism $\sigma(x)\in \PGL_{2}(\mathbb{Q})$ which permutes the roots of $f$. Suppose furthermore the $\sigma(x)$ does not have a pole at $\infty$. Let $D(T)=f(\frac{T^2}{f(\sigma(\infty))})$. Define a surface $\mathcal{X}_{D}^{f}\rightarrow \PPP^{1}$ by
$$\mathcal{X}_{D}^{f}: D(T)y^{2}=f(x)$$
and assume the Jacobian of its generic fiber $X_{D}^{f}/\Q(T)$ has trivial $\mathbb{Q}(T)/\Q$-trace. Then the following hold.
\begin{enumerate}
\item
If $f(x)$ has degree $3$, Nagao's conjecture is true for $\mathcal{X}^{f}_{D}$.
\item
If $f(x)$ has degree $5$, the Sato-Tate conjecture for abelian surfaces implies Nagao's conjecture for $\mathcal{X}^{f}_{D}$ (Conjecture \ref{HPconj}).
~\\
 Moreover, if $f(x)$ is a self-reciprocal (palindromic) polynomial, we can choose $\sigma(x)=\frac{1}{x}$. Then Nagao's conjecture is true for the surface
$$f(T^2)y^{2}=f(x).$$  
\end{enumerate}
\end{Theorem} 

\begin{proof}
By Theorem \ref{thm16}(3), we have $a_{p}(J_{D})=2a_{p}(J_{f})$ Thus we have
$$\sum^{p-1}_{t=0}\left(\frac{D(t)}{p}\right)=2\sum^{p-1}_{t=0} \left(\frac{f(t)}{p}\right).$$
Therefore if we write the trace of Frobenius of the curve $y^2=f(X)$ by $\f_{p}$, we have
$$a_{p}(\mathcal{X}_{D,t}^{f})=\left(\frac{D(t)}{p}\right)\f_{p},$$
where $\mathcal{X}_{D,t}^{f}$ is the specialization of $\mathcal{X}_{D}^{f}$ at $T=t$. Thus we obtain
$$p\cdot A_{p}(\mathcal{X}_{D}^{f})=\sum_{t=0}^{p-1}a_{p}(\mathcal{X}_{D,t}^{f})=\sum_{t=0}^{p-1} \left(\frac{D(t)}{p}\right)\f_{p}=-2\f_{p}^{2}$$
and
$$\lim_{N \to \infty}\frac{1}{\pi(N)}\sum_{p<N}-A_{p}(\mathcal{X}_{D}^{f})=\lim_{N \to \infty}\frac{1}{\pi(N)}~2\sum_{p<N}\frac{\f^{2}_{p}}{p},$$ 
which corresponds to twice of the second moment of the normalized trace of Frobenius in Table~\ref{my-label}, which is equal to the rank of endomorphism ring of $J_{f}$. Thus by using Theorem \ref{thm16}(2) we can conclude that the Sato-Tate conjecture for abelian surfaces implies Nagao's conjecture for the family.

\end{proof}

Moreover, we can generalize to the higher genus case assuming a nice factorization of the relevant Jacobian.

\begin{Theorem}
\label{thm21}
Let $y^{2}=f(x)$ be an elliptic curve or a hyperelliptic curve of genus $2$ defined over $\Q$. We consider a surface $\mathcal{X}^{f}_{D}\rightarrow \PPP^{1}$ which is defined using a hyperelliptic curve $y^2=D(T)\in\Q[T]$
$$\mathcal{X}^{f}_{D}: D(T)y^{2}=f(x).$$
Denote by $J_{f}$ and $J_{D}$ the Jacobian defined with the curves $y^2=f(x)$ and $y^2=D(T)$ respectively. Also, denote the generic fiber of $\mathcal{X}^{f}_{D}$ by $X^{f}_{D}/\Q(T)$, and we assume that $X^{f}_{D}/\Q(T)$ has trivial $\Q(T)/\Q$-trace.    
Assume that the Jacobian of $D$ has a ($\Q$-isogenous) factorization as $r$ copies of $J_{f}$, i.e., $J_{D}\sim J_{f}^r$. Then the following hold.
\begin{enumerate}
\item
If $y^{2}=f(x)$ is an elliptic curve, Nagao's conjecture is true for $\mathcal{X}^{f}_{D}$.
\item
If $y^{2}=f(x)$ is a hyperelliptic curve of genus $2$, the Sato-Tate conjecture for abelian surfaces implies Nagao's conjecture for $\mathcal{X}^{f}_{D}$ (Conjecture \ref{HPconj}).
\end{enumerate}
\end{Theorem}

Note that Theorem \ref{thm20} is a special case of Theorem \ref{thm21} when $r=2$.

\begin{proof}
Since isogenous abelian varieties have the same $L$-polynomial, from the factorization $J_{D}\sim J_{f}^r$, we can deduce
$$L_{p}(J_{D},T)=L_{p}(J_{f},T)^{r}.$$
Thus we know the trace of Frobenius of $J_{D}$ and $J_{f}$, $a_{p}(J_{D})$ and $\f_{p}$ respectively, satisfy the following relation
$$a_{p}(J_{D})=r\cdot \f_{p},$$
which implies the equality
$$\sum_{t-0}^{p-1}\left(\frac{D(t)}{p}\right)=r\sum_{t=0}^{p-1}\left(\frac{f(t)}{p}\right)=-r\cdot \f_{p}.$$
On the other hand, we can write the trace of Frobenius for each fiber of $\mathcal{X}^{f}_{D,t}$ as
$$a_{p}(\mathcal{X}^{f}_{D,t})=\left(\frac{D(t)}{p}\right)\cdot \f_{p},$$
thus we have
$$\sum_{t=0}^{p-1}a_{p}(\mathcal{X}^{f}_{D,t})=\sum_{t=0}^{p-1}\left(\frac{D(t)}{p}\right)\cdot \f_{p}=-r\cdot \f_{p}^2,$$
using the factorization of $J_{D}$.
Thus if $y^2=f(x)$ is an elliptic curve (if say, defined over a number field $k$),
\begin{displaymath}
 \lim_{N \to \infty}\frac{1}{\pi(N)}\sum_{p<N}-A_{p}(\mathcal{X}^{f}_{D}) = \left\{ 
\begin{array}{ll} 
2r & \textrm{if $E$ has CM and $k$ contains the field of CM,}\\ 
r & \textrm{otherwise.}
\end{array} 
\right. 
\end{displaymath}
\noindent
When $y^2=f(x)$ is a hyperelliptic curve of genus $2$, we get

\begin{displaymath}
 \lim_{N \to \infty}\frac{1}{\pi(N)}\sum_{p<N}-A_{p}(\mathcal{X}^{f}_{D}) =r\cdot \left(
\begin{tabular}{c} Second moment of normalized \\ trace of Frobenius $f_{p}/\sqrt{p}$
\end{tabular} \right),
\end{displaymath}
\noindent
which is shown in the third column of Table \ref{my-label} for each Sato-Tate type.
~\\

On the other hand, for a number field $k$, we know from the Proposition \ref{prop2.3}
$$J(X^{f}_{D})(k(T))\sim \Hom_{k}(J_{D},J_{f}) \sim \Hom_{k}(J_{f}^r,J_{f}) \sim \Hom_{k}(J_{f},J_{f})^{r}.$$
Thus we can conclude
$$\rank J(X^{f}_{D})(k(T))=r\cdot \rank \End_{k}(J_{f}), $$
which is equal to the second moment of the normalized trace of Frobenius for both cases with $k=\Q$, i.e., for hyperelliptic $y^2=f(x)$ case, it's $r$ times the second column of Table \ref{my-label}. Therefore the following equality holds
$$\lim_{N \to \infty}\frac{1}{\pi(N)}\sum_{p<N}-A_{p}(\mathcal{X}^{f}_{D})=\rank J(X^{f}_{D})(\Q(T))$$
and by Lemma \ref{lem1}, Nagao's conjecture is true
$$\lim_{N \to \infty}\frac{1}{N}\sum_{p<N}-A_{p}(\mathcal{X}_{D,t}^{f})\log p=\rank J(X^{f}_{D})(\Q(T)).$$

\end{proof}

\begin{Remark}
The Sato-Tate conjecture for abelian surfaces (Conjecture \ref{refined}) over number fields which are $\overline{\Q}$-isogenous to the square of an elliptic curve with complex multiplication is proven by F. Fit\'e and A. Sutherland \cite{FS}, and C. Johannson \cite{JHN}. The first $18$ lines in Table \ref{my-label} are the corresponding cases. Moreover, the conjecture is also true for the next $5$ lines, i.e., $E_{1}$, $E_{2}$, $E_{3}$, $E_{4}$, $E_{6}$, and also for $F_{ac}$ and $F_{a,b}$ by \cite{JHN}.   
\end{Remark}

\begin{Remark}
The restriction on the genus of $y^{2}=f(x)$ in Theorem \ref{thm20} and Theorem \ref{thm21} is essentially because the Sato-Tate conjecture for genus greater than $2$ has not been studied yet. 
\end{Remark}

\begin{Theorem}
Let $E: y^2=f(x)$ be an elliptic curve over $\Q$. We can consider a surface $\mathcal{X}^{f}_{D}\rightarrow \PPP^{1}$ which is defined using a hyperelliptic curve $y^2=D(T)\in\Q[T]$
$$\mathcal{X}^{f}_{D}: D(T)y^{2}=f(x).$$
Denote by $J_{f}$ and $J_{D}$ the Jacobian defined with the curves $y^2=f(x)$ and $y^2=D(T)$ respectively. Also, denote the generic fiber of $\mathcal{X}^{f}_{D}$ by $X^{f}_{D}/\Q(T)$, and we assume that $X^{f}_{D}/\Q(T)$ has trivial $\Q(T)/\Q$-trace.  
Assume that the Jacobian of $D$ has a ($\Q$-isogenous) factorization $J_{D}\sim E^r\times E_{1} \times E_{2} \times \cdots \times E_{s} $,
where $E_{i}$ are elliptic curves without complex multiplication which are not isogenous to $E$ for $i=1,2,\cdots ,s$. Then Nagao's conjecture is true for $\mathcal{X}^{f}_{D}$.
\end{Theorem}

\begin{proof}
From the factorization $J_{D}\sim E^r\times E_{1} \times E_{2} \times \cdots \times E_{s} $, we can deduce the factorization of the $L$-polynomial of $J_{D}$
$$L_{p}(J_{D},T)=L_{p}(E,T)^{r}\cdot L_{p}(E_{1},T)\cdot L_{p}(E_{2},T)\cdots L_{p}(E_{s},T).$$
Thus we get a relation of the trace of Frobenius $a_{p}(E)$, $a_{p}(E_{i})$ of each $E$ and $E_{i}$ for $i=1,\cdots , s$ by comparing the linear coefficient of the $L$-polynomials.
$$a_{p}(J_{D})=r\cdot a_{p}(E)+a_{p}(E_{1})+a_{p}(E_{2})+\cdots + a_{p}(E_{s})$$
and
$$\sum_{t=0}^{p-1}\left(\frac{D(t)}{p}\right)=-r\cdot a_{p}(E)-a_{p}(E_{1})-\cdots -a_{p}(E_{s}).$$
On the other hand, we can write the trace of Frobenius of each fiber of $\X^{f}_{D,t}$ as

$$a_{p}(\X_{D,t}^{f})=\left(\frac{D(t)}{p}\right)\cdot a_{p}(E),$$
so 
\begin{align}
\sum_{t=0}^{p-1}a_{p}(\X^{f}_{D,t})&=\sum_{t=0}^{p-1}\left(\frac{D(t)}{p}\right)\cdot a_{p}(E)\\
&=-r\cdot a_{p}(E)^{2}-a_{p}(E)\cdot a_{p}(E_{1})-a_{p}(E)\cdot a_{p}(E_{2})-\cdots - a_{p}(E)\cdot a_{p}(E_{s}).
\end{align}
We know two nonisogenous elliptic curves without complex multiplication have the product Sato-Tate distribution \cite[Proposition 2.1]{MP} which leads to
\begin{align*}
\label{eq21}
\lim_{N \to\infty}\frac{1}{\pi(N)}\sum_{p<N} & -A_{p}(\X^{D}_{f,t})\\
&= \lim_{N\to \infty}\frac{1}{\pi(N)}\sum_{p<N} \left(r\cdot \frac{a_{p}(E)^{2}}{p}+\frac{a_{p}(E)}{\sqrt{p}}\cdot \frac{a_{p}(E_{1})}{\sqrt{p}} +\cdots+\frac{a_{p}(E)}{\sqrt{p}}\cdot \frac{a_{p}(E_{s})}{\sqrt{p}}\right)\\
&=r\cdot \int^{\pi}_{0}4\cos^{2}\theta\cdot\frac{2}{\pi}\sin^{2}\theta d\theta+\int^{\pi}_{0}4\cos\theta\cos\theta_{1}\cdot\frac{4}{\pi^{2}}\sin^{2}\theta\sin^{2}\theta_{1}\cdot d\theta d\theta_{1} \\ &~~~~~~~~~~~~~~~~~~~~~~~~~~~~~~~~~~~~~~~~~~~~~ +
\cdots+\int^{\pi}_{0}4\cos\theta\cos\theta_{s}\cdot\frac{4}{\pi^{2}}\sin^{2}\theta\sin^{2}\theta_{s}\cdot d\theta d\theta_{s}\\
&=r+0+\cdots+0\\
&=r.
\end{align*}
Further, since the $E_{i}$ are nonisogenous to $E$ for all $i=1,\cdots s$, 
\begin{align*}
J(X^{f}_{D})(\Q(T))&\sim \Hom(J_{D},E)\\
& \sim \Hom(E^{r}\times E_{1}\times\cdots\times E_{s},E)\\ 
& \sim \Hom(E,E)^{r}\times \Hom(E_{1},E)\times\cdots\times \Hom(E_{s},E)\\
& \sim \Hom(E,E)^{r}.
\end{align*}
Thus we can conclude
$$\rank J(X^{f}_{D})(\Q(T))=r\cdot\rank\End(E)=r,$$
which tells us
$$\lim_{N \to\infty}\frac{1}{\pi(N)}\sum_{p<N}-A_{p}(\X^{D}_{f,t})=\rank J(X^{f}_{D})(\Q(t)).$$
Therefore, Nagao's conjecture is true by Lemma \ref{lem1}, i.e.,
$$\lim_{N \to \infty}\frac{1}{N}\sum_{p<N}-A_{p}(\mathcal{X}_{D,t}^{f})\log p=\rank J(X^{f}_{D})(\Q(T)).$$
\end{proof}

\section*{Acknowledgement}
The author would like to thank her advisor Joseph H. Silverman for
his continuously helpful advice. We would also like to thank Kiran Kedlaya, Ram Murty, Koh-ichi Nagao, and Andrew Sutherland for their helpful advice.

\addcontentsline{toc}{chapter}{\numberline{}Bibliography}

\nocite{CHT}
\nocite{CON}
\nocite{FKRS}
\nocite{FS}
\nocite{HP}
\nocite{HST}
\nocite{IR}
\nocite{KS}
\nocite{JHN}
\nocite{LANG}
\nocite{LN}
\nocite{MP}
\nocite{NAG}
\nocite{NER}
\nocite{PET}
\nocite{RS}
\nocite{RuSil}
\nocite{SER}
\nocite{SER1}
\nocite{SIL1}
\nocite{SIL2}
\nocite{TATE1}
\nocite{TAY}

\end{document}